\documentclass[11pt]{amsart}

\usepackage{latexsym}
\usepackage{amsmath}
\usepackage{amssymb}

\title[Linear Artin Approximation]{On The Linearity Of Artin Functions}
\author{Trung T. Dinh}

\newtheorem{theorem}{Theorem}[section]
\newtheorem{lemma}[theorem]{Lemma}

\theoremstyle{definition}
\newtheorem{definition}[theorem]{Definition}
\newtheorem{remark}[theorem]{Remark}

\newcommand{\rad}{\text{rad}}
\newcommand{\supp}{\text{supp}}

\newcommand{\mbf}{\mathbf}
\newcommand{\mf}{\mathfrak}

\begin{document}
\maketitle


\begin{abstract} It was proved by Elkik that, under some smoothness conditions, the Artin
functions of systems of polynomials over a Henselian pair are
bounded above by linear functions. This paper gives a stronger form
of this result for the class of excellent Henselian local rings. The
linearity of Artin functions of systems of polynomials in one
variable is also studied. Explicit calculations of Artin functions
of monomials and determinantal ideals are also included.
\end{abstract}


\section{Introduction}


Throughout this paper, all rings are Noetherian commutative with identity.
Let $R$ be such a ring and $\mathbf{f}=(f_1,\ldots,
f_r)$ a system of polynomials in $R[X_1,\ldots,X_N]$. For an
$R$-algebra $S$ and an element $\mbf{a}=(a_1,\ldots,a_N)\in S^N$,
we use the notation $\mathbf{f}(\mbf{a})$ to denote the set
$\{f_1(\mbf{a}),\ldots, f_r(\mbf{a})\}$ of elements in $S$. By
abuse of notation we shall often identify $\mbf{f}$ with the ideal
generated by $f_1,\ldots,f_r$, and $\mbf{f(a)}$ with the ideal
generated by $f_1(\mbf{a}),\ldots,f_r(\mbf{a})$ when it is clear
that there is no room for confusion.


\begin{definition}({\it Artin approximation property}) Let $R$
be a commutative ring, $I\subset R$ an ideal. Let $\widehat{R}$ be
the completion of $R$ with respect to the $I$-adic topology. Let
$\mbf{f}=(f_1,\ldots,f_r)\subseteq R[X_1,\ldots,X_N]$ be a system of
polynomials. We say that the system $\mbf{f}$ has the Artin
approximation property (AP) with respect to $(R,I)$ if for each
positive integer $c$ and for each
$\mbf{\widehat{a}}=(\widehat{a}_1,\ldots,\widehat{a}_N)\in
\widehat{R}^N$ with $\mbf{f}(\widehat{\mbf{a}})=0$, there is
$\mbf{a}=(a_1,\ldots,a_N)\in R^N$ such that $\mbf{f}(\mbf{a})=0$ and
$\mbf{\widehat{a}}\equiv\mbf{a}$ (mod $I^c\widehat{R}^N$).

If all systems of polynomials over $R$ have AP then we say that
$(R,I)$ has AP, or that $(R, I)$ is an approximation ring. When
$(R,\mathfrak{m})$ is a local ring, we will say that $R$ has AP if
$(R,\mathfrak{m})$ has.
\end{definition}


It is well-known that to show that the ring $R$ has AP it suffices to show
that  any system of polynomial equations over $R$
is solvable whenever it is solvable over $\widehat{R}$.


\begin{definition}({\it Strong Artin approximation property and
Artin functions}) Let $R, I, \widehat{R}$ and $\mbf{f}$ be as in the
above definition. We say that the system $\mbf{f}$ has the strong
Artin approximation property (SAP) with respect to $(R,I)$ if
there exists a sequence $\{\beta_n\}_{n\geq 1}$ of positive integers with
the following property. For each positive integer $n$ and for each
$\mbf{a}=(a_1,\ldots,a_N)\in R^N$ such that
$\mbf{f}(\mbf{a})\subseteq I^{\beta_n}$,
there is $\mbf{b}=(b_1,\ldots,b_N)\in R^N$ such that $\mbf{f}(\mbf{b})=0$
\linebreak and  $\mbf{a}\equiv\mbf{b} \text{ (mod } I^nR^N)$.
\end{definition}


If all systems of polynomials over $R$ have SAP then we say $(R,I)$
has SAP. When $(R,\mathfrak{m})$ is a local ring, we will say that
$R$ has SAP if $(R,\mathfrak{m})$ has. If for each $n\in \mathbb{N}$
we choose $\beta_n$ to be smallest possible then the function
$\beta(\mbf{f})\colon \mathbb{N}\rightarrow\mathbb{N}, n\mapsto \beta_n$, is called the Artin function
(associated to the system
$\mbf{f}$). This function, if exists, depends only on the
ideal generated by $f_1,\ldots,f_r$.

It has been proven
(\cite{artin1},\cite{artin2},\cite{artin3},\cite{ogoma},\cite{pfisterpopescu},\cite{popescu1},\cite{popescu2},\cite{rotthaus1},\cite{rotthaus2},\cite{spiva1},\cite{swan})
that a Noetherian local ring $(R,\mathfrak{m})$ has AP if and only
if it has SAP, and if and only if it is excellent and Henselian.

In \cite{spiva2} Spivakovsky announced a theorem, referred to as the
Linear Artin Approximation Theorem, without proof however, asserting
that for several classes of local rings (basically these are rings
for which the existence of Artin functions is verified), the
Artin function associated to any system of polynomial equations is
bounded above by a linear function. Since there was no proof
provided, the question on the linearity of Artin functions is then of interest.
Not much work in this direction however has
been published so far.

The first result was perhaps due to Greenberg (\cite{greenberg}) who
proved that Henselian discrete valuation rings, under some
separability conditions, have the linear approximation property,
though by that time this term had not been used. Elkik
(\cite{elkik}) studied this question in a more general set up,
namely on Noetherian Henselian pairs. During the course of the proof
of the main theorem in \cite{delfino} (see also \cite{delfino'}), the authors
proved the linearity of the Artin function for a certain system of a
single polynomial equation. Explicit calculations of linear bounds
for the Artin functions of hypersufaces with isolated singularity
over a power series ring in one variable appeared in papers by
Lejeune-Jalabert (\cite{lejeune}) and Hickel (\cite{hickel}).
However recently Rond, in \cite{rond1}, \cite{rond2},
provided two counterexamples to the above claim of Spivakovsky. Thus
the Artin functions in general are not bounded by linear functions.

This paper is concerned with the linearity of Artin functions. The
main result of the paper shows the linearity of the Artin functions under some smoothness conditions,
 which strengthens a result of Elkik for the class of excellent Henselian local rings. Its proof is carried
 out in Section 2. The third section deals with the Artin functions of systems of polynomials in one
 variable. Explicit calculations of Artin functions of monomial and determinantal
ideals will be studied in the last section.


\section{Main Theorem}


As mentioned in the introduction, this section is motivated by the
work of Elkik in \cite{elkik}. We need the following definitions in order to state the theorem of
 Elkik that inspired the main theorem of this section.


\begin{definition}(Elkik~\cite{elkik}) Let $R$ be a commutative
ring, $\mbf{f}=(f_1,\ldots,f_r)$ a system of polynomials in
$R[X_1,\ldots,X_N]$. Whenever $1\leq i_1 < \cdots < i_p\leq r$
denote by $\Delta (i_1,\ldots,i_p)$ the ideal generated by the
$p$-minors of the Jacobian matrix $\big (\frac{\partial
f_{i_j}}{\partial X_{k}}\big )_{j=1,\ldots,p}^{k=1,\ldots,N}$. Then we
define
$$
H_{\mbf{f}/R}=\sum_{1\leq i_1<\cdots < i_p\leq r}\Delta (i_1,\ldots, i_p)\big((f_{i_1},\ldots,f_{i_p}):_R(f_1,\ldots,f_r)\big).
$$
\end{definition}


It is well-known (\cite{elkik}) that $H_{\mbf{f}/R}$ is the smooth
locus of the quotient ring
$R[X_1,\ldots,X_N]/(\mbf{f})$ over $R$.

Recall that a pair $(R,I)$ with $I\subseteq \rad(R)$ is called a
Henselian pair if it satisfies the Hensel's Lemma. There are many
other equivalent definitions of Henselian pairs, and the one that is
most useful in the studies of Artin approximation is the following.


\begin{definition} A pair $(R, I)$ with $I\subseteq\rad(R)$ is called a Henselian pair if it satisfies the
following property: for each system of polynomials
$\mbf{f}=(f_1,\ldots,f_r)\subseteq R[X_1,\ldots,X_r]$, and for each
$\mbf{a}=(a_1,\ldots,a_r)\in R^r$ such that
$\mbf{f}(\mbf{a})\subseteq I$ and the determinant of the Jacobian
matrix $\big(\frac{\partial f_i}{\partial
X_j}(\mbf{a})\big)_{i,j=1,\ldots, r}$ is a unit, then there is
$\mbf{b}=(b_1,\ldots,b_r)\in R^r$ such that $\mbf{f}(\mbf{b})=0$ and
$\mbf{a}\equiv\mbf{b}$ mod $IR^r$.
\end{definition}


Now we can state the following theorem of Elkik.


\begin{theorem}\label{elkik2}\cite[Theorem 2, Section II]{elkik} Let $(R,I)$ be a Henselian pair. Then for
each integer $h$ there is a pair $(n_0,r)$ with the following
property. For each system of polynomials
$\mbf{f}=(f_1,\ldots,f_r)\subset R[X_1,\ldots,X_N]$, for each
$n>n_0$ and each $\mbf{a}=(a_1,\ldots,a_N)\in R^N$ satisfying
\begin{eqnarray*}
\mbf{f}(\mbf{a})&\subseteq& I^n\\
H_{\mbf{f}/R}(\mbf{a})&\supseteq& I^h,
\end{eqnarray*}
there is $\mbf{b}=(b_1,\ldots,b_N)\in R^N$ satisfying $\mbf{f}(\mbf{b})=0$
 and $\mbf{a}\equiv\mbf{b}$ mod $I^{n-r}R^N$.
\end{theorem}


Roughly speaking, this theorem assures the linearity of the Artin functions whenever the smooth
locus of the quotient ring $R[X_1,...,X_N]/(\mbf{f})$ is ``big''. By restricting ourselves to the
class of excellent Henselian local rings we obtain a stronger result, which is the main theorem of this section.


\begin{theorem}\label{main}
Let $(R,\mathfrak{m})$ be an excellent Henselian local ring of
dimension $d$ containing a field. Let $\mbf{f}=(f_1, \ldots, f_r)$
be a system of polynomials in $R[X_1,\ldots, X_N]$. Let
$t_1,\ldots,t_{d-1}$ be elements in $R$ such that  $t_i$ is not in
any minimal prime of $(t_1,\ldots,t_{i-1})$ for all $i\geq 1$.  Then
there is a linear function $\beta:
\mathbb{N}\longrightarrow\mathbb{N}$ with the following property:
for each $n\in\mathbb{N}$ and each $\mbf{a}=(a_1,\ldots,a_N)\in R^N$
satisfying
\begin{eqnarray*}
\mbf{f}(\mbf{a})&\subseteq& \mathfrak{m}^{\beta (n)}\\
\mbf{f}(\mbf{a})+H_{\mbf{f}/R}(\mbf{a})&\supseteq&
(t_1,\ldots,t_{d-1}),
\end{eqnarray*}
there is $\mbf{b}=(b_1,\ldots,b_N)\in R^N$ such that $\mbf{f}(\mbf{b})=0$ and
$\mbf{a}\equiv \mbf{b}\text{ (mod }\mathfrak{m}^nR^N)$.
\end{theorem}


For the proof of this theorem we need some auxiliary lemmas. The first
lemma is concerned with the behavior of the Artin functions under
module-finite extensions. It is well-known that the approximation
property is preserved under module-finite extensions. The proof for
this fact can be easily adjusted to show that the same conclusion
holds for the strong approximation property. In what follows the Linear Approximation
property is abbreviated as LAP.


\begin{lemma}\label{ext} Let $S$ be a module-finite $R$-algebra, $I\subset R$ an ideal. Suppose $(R,I)$ has SAP.
Then $(S,IS)$ also has SAP. Furthermore, if $(R,I)$ has LAP then so does $(S,IS)$.
\end{lemma}


\begin{proof} We can write $S=Rs_1+\cdots + Rs_p$. Let $\phi$ be the $R$-linear
map $R^p\longrightarrow S$ defined by $\phi (e_i) = s_i$ for all
$i$, where the $e_i$ form the standard basis for $R^p$. Suppose
$\ker \phi\subseteq R^p$
is generated by $z_1=\begin{pmatrix}z_{11}\\ \vdots\\z_{1p}\end{pmatrix},\ldots, z_l=\begin{pmatrix}z_{l1}\\ \vdots\\
z_{lp}\end{pmatrix}, z_{ij}\in R$.

Let $F_1, \ldots, F_r$ be a system of polynomials in $S[X_1, \ldots, X_N]$. For each $i=1,\ldots,N$ we introduce new
variables $X_{i1}, \ldots, X_{ip}$ and write
$$
F_i(\sum_{j=1}^{p}X_{1j}s_j, \ldots, \sum_{j=1}^pX_{Nj}s_j)=\sum_{j=1}^pF_{ij}(X_{11}, \ldots, X_{Np})s_j,
$$
where the $F_{ij}$ are polynomials with coefficients in $R$. Define
new variables $T_{ij}, i=1,\ldots, r, j=1,\ldots, l$. Let $\beta $
be a linear bound for the Artin function of the system
$\{G_{ik}=F_{ik}(X_{11}, \ldots, X_{Np})-\sum_{j=1}^pT_{ij}z_{jk}|\quad
i=1,\ldots, r; k=1,\ldots, N\}$ of polynomials in $R[\{X_{ij}\},
\{T_{ij}\}]$. We claim that this function is also a bound for the
Artin function of the system $\{F_i\}$ in $S[X_1, \ldots, X_N]$.

Suppose for some $\{a_i\}\subset S $ we have $F_i(a_1, \ldots, a_N)\in I^{\beta (n)}S$ for all $i$. We can write
$a_i=\sum_{j=1}^pa_{ij}s_j$ with $a_{ij}\in R$. Then
\begin{eqnarray*}
\sum_{j=1}^pF_{ij}(a_{11}, \ldots, a_{Np})s_j&\in& I^{\beta (n)}S\\
\Longrightarrow\begin{pmatrix} F_{i1}(a_{11},\ldots,a_{Np})\\ \vdots \\
F_{ip}(a_{11},\ldots,a_{Np})\end{pmatrix} &\in&I^{\beta (n)}R^p+\ker\phi\\\Longrightarrow\begin{pmatrix}
F_{i1}(a_{11},\ldots,a_{Np})\\ \vdots \\ F_{ip}(a_{11},
\ldots,a_{Np})\end{pmatrix}-\sum_{j=1}^lu_{ij}\begin{pmatrix}z_{j1}\\ \vdots\\
z_{jp}\end{pmatrix}&\in& I^{\beta (n)}R^p
\end{eqnarray*}
for some $u_{ij}\in R$.

This implies that $\{a_{ij}\}, \{u_{ij}\}$ is an approximate solution of order $\beta (n)$ of the system
$\{G_{ik}\}$. Then there are$\{b_{ij}\}, \{v_{ij}\}$ consisting of
 elements of $R$ such that $a_{ij}\equiv b_{ij}, u_{ij}\equiv v_{ij}$ mod $I^n$ and $F_{ik}(\{b_{ij}\})=\sum_{j=1}^lv_{ij}z_{jk}$
 for all $i=1,\ldots, r; k=1,\ldots, p$.
 Define $b_i=\sum_{j=1}^pb_{ij}s_j$. Clearly we have $a_i\equiv b_i$ mod $I^nS$ and

\begin{eqnarray*}
\begin{pmatrix} F_{i1}(b_{11},\ldots,b_{Np})\\ \vdots \\ F_{ip}(b_{11},\ldots,b_{Np})\end{pmatrix}&=&\sum_{j=1}^lv_{ij}\begin{pmatrix}z_{j1}\\ \vdots\\
z_{jp}\end{pmatrix}\in\ker \phi\\
\Longrightarrow
F_i(b_1,\ldots,b_N)&=&\sum_{j=1}^pF_{ij}(b_{11},\ldots,b_{Np})s_j=0.
\end{eqnarray*}
\end{proof}


The next easy lemma often helps us to reduce the problems to the complete case.


\begin{lemma}\label{complete}
Let $(R,\mathfrak{m})$ be a local ring, $\mbf{f}=(f_1,\ldots,f_r)$ a
system of polynomials in $R[X_1,\ldots, X_N]$. Suppose this system
has AP. If this system, viewed as polynomials over $\widehat{R}$,
has SAP then it has SAP as polynomials over R. The conclusion also
holds if we replace SAP by LAP.
\end{lemma}


\begin{proof} View $(f_1,\ldots,f_r)$ as polynomials in
$\widehat{R}[X_1,\ldots,X_N]$. Let $\beta$ be a bound for the
Artin function of this system with respect to
$(\widehat{R},\widehat{\mathfrak{m}})$. We claim that this function
also works for the system viewed as polynomials in $R[X_1,\ldots,X_N]$.
Indeed, suppose there is $\mbf{a}=(a_1,\ldots,a_N)\in R^N$ such that
$\mbf{f}(\mbf{a})\subseteq \mathfrak{m}^{\beta (n)}\subseteq
\widehat{\mathfrak{m}}^{\beta (n)}.$ Then there is
$\widehat{\mbf{a}}=(\widehat{a}_1,\ldots,\widehat{a}_N)\in
\widehat{R}^N$ such that $\mbf{a}\equiv \widehat{\mbf{a}}$ (mod
$\widehat{\mathfrak{m}}^n$) and $\mbf{f}(\widehat{\mbf{a}})=0.$
Since $\mbf{f}$ has AP over $R$ by the hypothesis, we can find
$\mbf{b}=(b_1,\ldots,b_N)\in R^N$ such that
$\mbf{b}\equiv\widehat{\mbf{a}}$ (mod $\widehat{\mathfrak{m}}^n$)
and $\mbf{f}(\mbf{b})=0$.
\end{proof}


The proof of the main theorem uses induction on the dimension of the ring, for which
we will need the following result on the rings of dimension one.


\begin{theorem}\label{dim1}
Let $(R,\mathfrak{m})$ be an excellent Henselian local ring of
dimension 1 which is assumed to be an integral domain in the
mixed characteristic case. Then $R$ has LAP.
\end{theorem}


\begin{proof} We know that with the given hypothesis $R$ has AP. By
Lemma~\ref{complete} it suffices to assume that $R$ is complete.
Using the Cohen structure theorem there is a coefficient ring of $R$
over which $R$ is a finite module. Since $R$ has dimension 1, this
coefficient ring must be a complete DVR, denoted by $(V,p)$. By a theorem of
Greenberg (\cite{greenberg}), $(V,p)$ has LAP. Thus $(R,pR)$
has LAP by Lemma ~\ref{ext}. But the ideal $pR$ is $\mathfrak{m}$-primary, so $(R,\mathfrak{m})$
also has LAP.
\end{proof}


The following lemma tells us the behavior of the smooth locus after adjoining a certain equation to the system.


\begin{lemma}\label{enlarge}
Let $\mbf{X}=(X_1,\ldots,X_N)$ and $\mbf{f}=(f_1,\ldots,f_r)\subseteq
R[\mbf{X}]$. Suppose $H_{\mbf{f}/R}$ is generated by
$g_1,\ldots,g_M$ in $R[\mbf{X}]$. Let $\mbf{Y}=(Y_1,\ldots,Y_M),
\mbf{Z}=(Z_1, \ldots, Z_T)$ be new variables and $h\in
R[\mbf{X},\mbf{Z}]$. Define $G=h+Y_1g_1+\cdots+Y_Mg_M\in
R[\mbf{X},\mbf{Y},\mbf{Z}]$. Then
$$
H_{\mbf{f}, G/R}\supseteq \big(H_{\mbf{f}/R}\big)^2
R[\mbf{X},\mbf{Y},\mbf{Z}].
$$
\end{lemma}


\begin{proof}
Whenever $1\leq i_1<\cdots<i_p\leq r$ denote by
$\Delta_{i_1,\ldots,i_p}^G$ the ideal generated by the $(p+1)$-minors of
the Jacobian matrix of the system $f_{i_1},\ldots,f_{i_p},G$. This
matrix can be explicitly calculated as
$$
\begin{pmatrix}
\frac{\partial f_{i_1}}{\partial X_1} & \ldots & \frac{\partial
f_{i_p}}{\partial X_1}& \frac{\partial
G}{\partial X_1}\\
. & \ldots &. & \ldots\\
\frac{\partial f_{i_1}}{\partial X_N} & \ldots & \frac{\partial
f_{i_p}}{\partial X_N}& \frac{\partial G}{\partial X_N}\\
0 & \ldots & 0 & g_1\\
. & \ldots & . & .\\
0 & \ldots & 0 & g_M\\
0 & \ldots & 0 & \frac{\partial G}{\partial Z_1}\\
. & \ldots & . & .\\
0 & \ldots & 0 & \frac{\partial G}{\partial Z_r}
\end{pmatrix}
$$

Consider the $(p+1)$-minor of this matrix formed by the rows
corresponding to $X_{j_1},\ldots, X_{j_p}$ and $g_k$ for some $1\leq
j_1 <\cdots<j_p\leq N, 1\leq k\leq M$. This minor is clearly equal
to the product of $g_k$ and the $p$-minor of the Jacobian matrix (in
$R[\mbf{X}]$) of $f_{i_1},\ldots,f_{i_p}$ formed by the rows
corresponding to $X_{j_1},\ldots,X_{j_p}$. We can deduce from this
fact that
\begin{eqnarray*}
\Delta_{i_1,\ldots,i_p}^G&\supseteq& \Delta_{i_1,\ldots,i_p}(g_1,\ldots,g_M)R[\mbf{X},\mbf{Y},\mbf{Z}]\\
&=& \Delta_{i_1,\ldots,i_p}H_{\mbf{f}/R}R[\mbf{X},\mbf{Y},\mbf{Z}].
\end{eqnarray*}
Hence
\begin{eqnarray*}
H_{\mbf{f},G/R}&\supseteq &\sum_{1\leq i_1<\cdots<i_p\leq r}\Delta_{i_1,\ldots,i_p}^G\big((f_{i_1},\ldots,f_{i_p},G):_R(f_1,\ldots,f_r,G)\big)\\
&\supseteq& \sum_{1\leq i_1<\cdots<i_p\leq r} \Delta_{i_1,\ldots,i_p}^G \big((f_{i_1},\ldots,f_{i_p}):_R(f_1,\ldots,f_r)\big)\\
&\supseteq & \sum_{1\leq i_1<\cdots<i_p\leq
r}\Delta_{i_1,\ldots,i_p}\big((f_{i_1},\ldots,f_{i_p}):_R(f_1,\ldots,f_r)\big)
H_{\mbf{f}/R}R[\mbf{X},\mbf{Y},\mbf{Z}]\\
&=& \big(H_{\mbf{f}/R}\big)^2R[\mbf{X},\mbf{Y},\mbf{Z}].
\end{eqnarray*}
\end{proof}


We will also need the following lemma of Elkik.


\begin{lemma}\label{elkik1}\cite[Lemma 1, Section I]{elkik} Let $R$ be a complete ring
with respect to the $(t)$-adic topology for some $t\in R$.
Then for each integer $h$ there is an integer $n_0$ with the
following property. For any ideal $I$ of $R$,
$\mbf{f}=(f_1,\ldots,f_r)\subset R[X_1,\ldots,X_N]$, for all $n>n_0$
and for all $\mbf{a}=(a_1,\ldots,a_N)\in R^N$ with
\begin{eqnarray*}
\mbf{f}(\mbf{a})&\subseteq& t^nI\\
H_{\mbf{f}/R}(\mbf{a})&\supseteq& (t^h),
\end{eqnarray*}
there is $\mbf{b}=(b_1,\ldots,b_N)\in R^N$ satisfying $\mbf{f}(\mbf{b})=0$
 and $\mbf{a}\equiv\mbf{b}$ mod $t^{n-h}IR^N$.
\end{lemma}


We are now able to prove the main theorem.


\begin{proof}(of Theorem \ref{main}) We may assume that $R$ is complete. Indeed, suppose the theorem is
proven for the complete local ring
$(\widehat{R},\mathfrak{m}\widehat{R})$. Then, in particular, there
is a function $\beta:\mathbb{N}\rightarrow\mathbb{N}$ satisfying:
for each $n\in\mathbb{N}$, each $\mbf{a}\in R^N$ such that
\begin{eqnarray*}
\mbf{f}(\mbf{a})&\subseteq& \mathfrak{m}^{\beta (n)}\widehat{R}\\
\mbf{f}(\mbf{a})+H_{\mbf{f}/\widehat{R}}(\mbf{a})&\supseteq&
(t_1,\ldots,t_{d-1})\widehat{R},
\end{eqnarray*}
there is $\widehat{\mbf{b}}\in\widehat{R}^N$ with
$\mbf{f}(\widehat{\mbf{b}})=0$ and $\mbf{a}\equiv\widehat{\mbf{b}}$
(mod $\mathfrak{m}^n\widehat{R}^N$). We claim that the same function
$\beta$ satisfies the conclusion of the theorem. Let $\mbf{a}\in
R^N$ be such that
\begin{eqnarray*}
\mbf{f}(\mbf{a})&\subseteq& \mathfrak{m}^{\beta (n)}R\\
\mbf{f}(\mbf{a})+H_{\mbf{f}/R}(\mbf{a})&\supseteq&
(t_1,\ldots,t_{d-1})R.
\end{eqnarray*}
Since $H_{\mbf{f}/\widehat{R}}\supseteq H_{\mbf{f}/R}\widehat{R}$,
the above relations remain true if we replace $R$ with
$\widehat{R}$. By the choice of $\beta$ there is
$\widehat{\mbf{b}}\in\widehat{R}^N$ with
$\mbf{f}(\mbf{\widehat{b}})=0$ and $\mbf{a}\equiv\widehat{\mbf{b}}$
(mod $\mathfrak{m}^n\widehat{R}^N$). Because $R$ is excellent and
Henselian, hence has AP, we can find $\mbf{b}\in R^N$ with
$\mbf{f}(\mbf{b})=0$ and $\mbf{a}\equiv \mbf{b}$ (mod
$\mathfrak{m}^n\widehat{R}^N$), hence $\mbf{a}\equiv \mbf{b}$ (mod $\mathfrak{m}^nR^N$).

We prove the theorem by induction on the dimension of $R$. The case $\dim R=0$ is trivial. For the case $\dim
R=1$ the result is clear due to Theorem~\ref{dim1}. Suppose $\dim R
>1$. We may assume that $t_1\in\mathfrak{m}$, otherwise the ideal $(t_1,\ldots,t_{d-1})$ will be the whole ring $R$ and then the theorem will follow from Theorem \ref{elkik2}.

Apply Lemma~\ref{elkik1} with $t=t_1, h=1$. Let $n_0$ be as in the
statement of that lemma and
 choose $s\geq n_0+2$. By the Artin-Rees Lemma there is $k$ such that $(t_1)^s\cap \mathfrak{m}^{n+k}\subseteq
(t_1)^s\mathfrak{m}^{n}$ for all $n>0$.

Let $(-)$ denote passing modulo $(t_1)^s$. Suppose
$H_{\mbf{f}/R}=(g_1,\ldots,g_M), g_i\in R[\mbf{X}].$ Introduce new
variables $\mbf{Y}=(Y_1,\ldots,Y_M), \mbf{Z}=(Z_1,\ldots,Z_r)$ and let
$$
G(\mbf{X},\mbf{Y},\mbf{Z})=Z_1f_1+\cdots+Z_rf_r+Y_1g_1+\ldots+Y_Mg_M-t_1\in R[\mbf{X},\mbf{Y},\mbf{Z}].
$$
Since $\dim\bar{R}<\dim R$, by the induction hypothesis there is a
linear function $\beta:\mathbb{N}\longrightarrow\mathbb{N}$
satisfying the following condition: for all $n\in \mathbb{N},
\mbf{a}\in R^N,\mbf{y}\in R^M, \mbf{z}\in R^r$ such that
\begin{eqnarray*}
\big(\mbf{\bar{f}}(\mbf{\bar{a}}),\bar{G}(\mbf{\bar{a}},\mbf{\bar{y}},\mbf{\bar{z}})\big)&\subseteq& \overline{\mathfrak{m}}^{\beta (n)}\\
\big(\mbf{\bar{f}}(\mbf{\bar{a}}),\bar{G}(\mbf{\bar{a}},\mbf{\bar{y}},\mbf{\bar{z}})\big)+H_{\mbf{\bar{f}},\bar{G}/\bar{R}}(\mbf{\bar{a}},\mbf{\bar{y}},\mbf{\bar{z}})&\supseteq&
(\bar{t}_2\ldots,\bar{t}_{d-1})^2,
\end{eqnarray*}
there is $\mbf{b}\in R^N,\mbf{u}\in R^M, \mbf{v}\in R^r$ such that
\begin{eqnarray*}
\mbf{\bar{f}}(\mbf{\bar{b}})&=&\bar{G}(\mbf{\bar{b}},\mbf{\bar{u}},\mbf{\bar{v}})=0\\
\mbf{\bar{a}}&\equiv& \mbf{\bar{b}}\text{ (mod }
\mathfrak{m}^n\bar{R}^N).
\end{eqnarray*}
We may assume that $\beta(n)>n$. Then we claim that the function $\beta_k:\mathbb{N}\rightarrow\mathbb{N},
\beta_k(n)=\beta (n+k)$ will satisfy the conclusion of the theorem.
Indeed, suppose for some $\mbf{a}\in R^N$ we have
\begin{eqnarray*}
\mbf{f}(\mbf{a})&\subseteq& \mathfrak{m}^{\beta (n+k)}\\
\mbf{f}(\mbf{a})+H_{\mbf{f}/R}(\mbf{a})&\supseteq&
(t_1,\ldots,t_{d-1}).
\end{eqnarray*}
The second inclusion implies that there is $\mbf{y}\in R^M,
\mbf{z}\in R^r$
 such that $G(\mbf{a},\mbf{y},\mbf{z})=0$. (Here we need $d>1$ which explains why we can't start the induction process from the dimension 0.) Thus by Lemma \ref{enlarge} we have relations
\begin{eqnarray*}
\big(\mbf{f}(\mbf{a}),G(\mbf{a},\mbf{y},\mbf{z})\big)&\subseteq& \mathfrak{m}^{\beta (n+k)}\\
\big(\mbf{f}(\mbf{a}),G(\mbf{a},\mbf{y},\mbf{z})\big)+H_{\mbf{f},G/R}(\mbf{a},\mbf{y},\mbf{z})&\supseteq&
(t_1,\ldots,t_{d-1})^2.
\end{eqnarray*}
Since $ H_{\mbf{\bar{f}},\bar{g}/\bar{R}} \supseteq
\overline{H}_{\mbf{f},g/R}$, all the relations here remain true
after replacing everything with its quotient modulo $(t_1)^s$. Thus
by the definition of $\beta$ there is $\mbf{b}^\prime\in R^N,
\mbf{u}\in R^M, \mbf{v}\in R^r$ such that
\begin{eqnarray*}
\mbf{\bar{a}}&\equiv& \mbf{\bar{b}}^\prime \text { mod } \mathfrak{m}^{n+k}\bar{R}^N\\
\mbf{f}(\mbf{\bar{b}}^\prime)&=&G(\mbf{\bar{b}}^\prime,\mbf{\bar{u}},\mbf{\bar{v}})=0.
\end{eqnarray*}
We may assume $\mbf{b}^\prime\equiv \mbf{a}$ mod
$\mathfrak{m}^{n+k}R^N$. The second relation of the above implies
$$
\mbf{f}(\mbf{b}^\prime)\subseteq (t_1)^s,
$$
and
$$
v_1f_1(\mbf{b}^\prime)+\cdots+v_rf_r(\mbf{b}^\prime)+u_1g_1(\mbf{b}^\prime)+\cdots+u_Mg_M(\mbf{b}^\prime)-t_1\in
(t_1)^s.
$$
It follows that
$$
(t_1)\subseteq
H_{\mbf{f}/R}(\mbf{b}^\prime)+\mbf{f}(\mbf{b}^\prime)+(t_1)^s\subseteq H_{\mbf{f}/R}(\mbf{b}^\prime)+(t_1)^s,
 $$
which implies $H_{\mbf{f}/R}(\mbf{b^\prime})\supseteq (t_1)$ using
Nakayama's lemma (note that $s\geq 2$). Since $\mbf{a}\equiv
\mbf{b}^\prime$ (mod $\mathfrak{m}^{n+k}R^N)$ we have
$$
f_i(\mbf{b}^\prime)\equiv f_i(\mbf{a}) \text{ (mod }\mathfrak{m}^{n+k}) \in \mathfrak{m}^{\beta (n+k)}\subseteq
\mathfrak{m}^{n+k}
$$
for all $i=1,\ldots,r$. Thus
$$
\mbf{f}(\mbf{b}^\prime)\subseteq (t_1)^s\cap\mathfrak{m}^{n+k}\subseteq (t_1)^s\mathfrak{m}^{n},
$$
by the choice of $k$. From Lemma~\ref{elkik1} it follows that there is $\mbf{b}\in R^N$ such that $\mbf{f}(\mbf{b})=0$
 and $\mbf{b}^\prime\equiv \mbf{b}$ mod $(t_1)^{s-1}\mathfrak{m}^{n}R^N\subseteq \mathfrak{m}^nR^N$.  Clearly $\mbf{b}\equiv \mbf{a}$
 (mod $\mathfrak{m}^nR^N$).
\end{proof}


\noindent\textbf{Example}. Let $(R,\mathfrak{m})$ be an excellent
Henselian local domain of dimension 2 containing a field (e.g
$R=k[[T_1, T_2]]$). Then the LAP holds for any single equation
$f=g(\mbf{X})+tY$, where $t\in R, t\neq 0, g\in R[\mbf{X}]$, since
$H_{(f)/R}\supset (t)$. Note that such a general statement can not
be true if we replace $Y$ by $Y^2$, due to a counterexample of G.
Rond in~\cite{rond1}.


\section{ Polynomials In One Variable}


In this section we study the linearity of the Artin functions of
systems of polynomial equations in only one variable over a reduced approximation ring. We would
 like to recall a well-known property of approximation rings that if an approximation (local) ring is
  reduced then its completion is also reduced.

We need some lemmas for the proof of the main result of this section.


\begin{lemma}\label{decomp}
Let $(R,\mathfrak{m})$ be an analytically irreducible local domain
and let $\mbf{f}=(f_1,\ldots,f_r)\subseteq R[X_1,\ldots,X_N]$.
Suppose for each $i\in\{1,\ldots,r\}$ the polynomial $f_i$ has a
decomposition $f_i=f_{i1}\cdots f_{ik_i}$ as a product of polynomials such that LAP holds for every system of
polynomials $(f_{1j_1},\ldots,f_{rj_r})$, $1\leq j_1 \leq
k_1,\ldots,1\leq j_r\leq k_r$. Then the LAP holds for the system
$\mbf{f}$.
\end{lemma}


\begin{proof}Let $\beta_{j_1\ldots j_r}$ be a linear bound for the Artin
function of the system $f_{1j_1},\ldots,f_{rj_r}$ for each $1\leq
j_1\leq k_1,\ldots, 1\leq j_r\leq k_r$. Let $k=\max_{i=1\ldots
r}\{k_i\}$. By \cite[Theorem 4.3]{swanson1}, there is a constant $c$
such that for every $n$, whenever $z_1z_2\cdots
z_k\in\mathfrak{m}^{cn}$ with $z_i\in R, i=1,\ldots,k,$ we can find
$i$ such that $z_i\in\mathfrak{m}^n$. We claim that the function
$\beta =c(\max\{\beta_{j_1\ldots j_r}\})$ is a linear bound for the
Artin function of the system $\mbf{f}$. Indeed, suppose for some
$n\in\mathbb{N}$ and some $\mbf{a}\in R^N$ we have $f_i(\mbf{a})\in
\mathfrak{m}^{c(\max\{\beta_{j_1\ldots j_r}\})(n)}$ for all $i$.
Since $f_i=f_{i1}\cdots f_{ik_i}$ and by the choice of $c$, for each
$i$ there is $j_i\in\{1,\ldots,k_i\}$ such that
$f_{ij_i}(\mbf{a})\in\mathfrak{m}^{\max\{\beta_{j_1\ldots
j_r}\}(n)}$. Hence the collection $\{j_1,\ldots,j_r\}$ has the
property that
$f_{ij_i}(\mbf{a})\in\mathfrak{m}^{\beta_{j_1\ldots j_r}(n)}$ for
all $i=1,\ldots, r$. By the definition of $\beta$ we can find
$\mbf{b}\in R^N$ such that $\mbf{a}\equiv\mbf{b}$ mod
$\mathfrak{m}^nR^N$, and $f_{ij_i}(\mbf{b})=0, i=1,\ldots, r$, which
also satisfies $f_i(\mbf{a})=0$ for all $i$.
\end{proof}
It was proved in \cite[Theorem 3.1]{rond2} that every system of homogeneous linear
equations over a Noetherian local ring has LAP. The same is true for any system of linear equations.


\begin{lemma}\label{linear}
Let $(R,\mathfrak{m})$ be a Noetherian local ring. Then every system
of linear equations over $R$ has LAP.
\end{lemma}
\begin{proof}Let $\mbf{f}=(f_i=c_{i1}X_1+\cdots+c_{iN}X_N+d_i|\quad
i=1,\ldots,r)$ be a system of linear equations over $R$. Let $Y$ be
a new variable and define $g_i=c_{i1}X_1+\cdots+c_{iN}X_N+d_iY,
i=1,\ldots,r$. Then $\mbf{g}=(g_1,\ldots,g_r)$ is a system of
homogeneous equations over $R$ in $N+1$ variables
$X_1,\ldots,X_N,Y$. By Theorem 3.1 in~\cite{rond2}, there is a
linear bound $\beta$ for the Artin function of this system. Then
this is also a bound for the Artin function of $\mbf{f}$. Indeed,
let $n\in\mathbb{N}$ and $\mbf{a}=(a_1,\ldots,a_N)\in R^N$ be such
that $\mbf{f}(\mbf{a})\subseteq\mathfrak{m}^{\beta(n)}$. Hence
$\mbf{g}(\mbf{a},1)\subseteq\mathfrak{m}^{\beta(n)}$. Then there is
$(b_1,\ldots,b_N,u)\in R^{N+1}$ such that $\mbf{g}(\mbf{b},u)=0$ and
$(\mbf{a},1)\equiv (\mbf{b},u)$ (mod $\mathfrak{m}^nR^{N+1}$). That
implies that $u$ is a unit in $R$. This finishes the proof of the
lemma since $\mbf{f}(u^{-1}\mbf{b})=0$ and $\mbf{a}\equiv
u^{-1}\mbf{b}$ (mod $\mathfrak{m}^nR^N$).
\end{proof}


The main result of this section is the following theorem whose proof is inspired by the work of
Rond in \cite{rond2}.


\begin{theorem}\label{onevariable}  Let $(R,\mathfrak{m})$ be an excellent Henselian reduced local
 ring. Then the Artin functions associated to systems of polynomials in one variable are bounded by linear functions.
\end{theorem}


\begin{proof} By the remark at the beginning of this section, we
may assume that $(R,\mf{m})$ is complete.
We first prove the theorem for the case $R$ is a domain, hence it
is analytically irreducible. Let $(f_1,\ldots,f_r)\subseteq R[X]$ be a
system of polynomials in one variable $X$. View them as polynomials
over the quotient field $Q(R)$ of $R$. Then for each $i$ the
polynomial $f_i$ has a decomposition $f_i=f_{i1}\cdots f_{ik_i}$ in
to irreducible polynomials over $Q(R)$. Since we have only a finite
number of polynomials, there is $E\in R, E\neq 0$ such that
$Ef_{ij}\in R[X]$ for all $i, j$. Suppose we can prove that each
system $(Ef_{1j_1},\ldots,Ef_{rj_r})$, $1\leq j_1 \leq
k_1,\ldots,1\leq j_r\leq k_r$ has LAP then by the lemma, the system
$E^{k_1}f_1,\ldots,E^{k_r}f_r$ has LAP, which clearly implies that
the system $(f_1,\ldots,f_r)$ has LAP. Therefore, we may assume that
each $f_i$ is irreducible in $Q(R)[X]$. It follows that for each $i$,
either $f_i$ is linear, or $f_i$ has no
solution in the quotient ring of $R$, in particular in $R$. If for
each $i=1,\ldots, r$ the polynomial $f_i$ is linear then the Artin
function of the system $\mbf{f}$ is bounded by a linear function by
Lemma \ref{linear}. If there is some $i$ such that $f_i$ has no
solution in $R$, then there is a positive integer $c$ such that for
all $a\in R$ we have $f_i(a)\not\in\mf{m}^c$. Indeed, if that number
$c$ does not exist, then since $R$ has the strong approximation
property, there would be a solution in $R$ for $f_i$ which is a
contradiction. Now it is clear that the constant function $\beta(n)=c$
is a bound for the Artin function of the system $\mbf{f}$.

Let us consider a complete reduced local ring $(R,\mf{m})$. Let
$P_1,\ldots,P_k$ be all minimal primes in $R$. Let
$\mbf{f}=(f_1,\ldots,f_r)\subseteq R[X]$. In $R/P_i$ the system
$\mbf{f}$ has either only finite number of roots or is identically
equal to 0. For each $i$ for which the system $\mbf{f}$ is not
identically equal to 0 in $R/P_i$ we consider a set of some elements
in $R$ whose images in $R/P_i$ form a complete set of all roots of
$\mbf{f}$. Let S be the union of all such sets. Then $S$ is finite.
For each $i$ let $P_i=(p_{i1},\ldots,p_{it})$.

By the previous part, there is a constant $c\in\mathbb{N}$
such that the linear function $n\mapsto cn, n\in\mathbb{N}$, is a
common linear bound for Artin functions of $\mbf{f}$ in $R/P_i$ for
every $i$. Let $Y_{ij}, i=1,\ldots,k, j=1,\ldots,t$ be new variables and
consider all possible systems of linear equations over $R$, each
equation is of the form $\{u_i+\sum_j p_{ij}Y_{ij}-u_l-\sum_j
p_{lj}Y_{lj}, u_i, u_l\in S|\quad j=1,\ldots,t\}$. The above lemma
assures that for each such system there exists a linear bound for its
Artin function. There are only finitely many such systems. Hence
we can find $d\in\mathbb{N}$ such that the function $n\mapsto dn,
n\in\mathbb{N}$, is a bound for the Artin function of every system
of this kind. We claim that the function $\beta\colon n\mapsto cdn,
n\in\mathbb{N}$, is a bound for the Artin function of the system
$\mbf{f}$ in $R$.

Let $a\in R$ be such that $\mbf{f}(a)\subseteq \mathfrak{m}^{cdn}$.
Passing modulo $P_i$ we get that, by the definition of $c$, for each $i$ there is $a_i\in S$
such that $\mbf{f}(a_i)\subseteq P_i$ and $a\equiv a_i$ (mod
$P_i+\mathfrak{m}^{dn})$. Then there is $\mbf{u}=(u_{ij})\in R^{kt}$
such that
$$a\equiv a_i+\sum_{j=1}^tp_{ij}u_{ij} \text{ (mod }\mathfrak{m}^{dn}).$$

Define
 $$L_{il}(\mbf{Y})=a_i+\sum_j p_{ij}Y_{ij}-a_l-\sum_j p_{lj}Y_{lj}$$
 for all $i=1,\ldots,k,l=1,\ldots,t$. Then $L_{il}(\mbf{u})\in \mathfrak{m}^{dn}$
 for all $i,l$. By the choice of $d$, there is
$\mbf{b}^\prime=(b_{ij}^\prime)\in R^{kt}$ such that
$L_{il}(\mbf{b}^\prime)=0$ and $\mbf{u}\equiv\mbf{b}^\prime$ (mod
$\mathfrak{m}^nR^{kt}$). Define
$$b=a_i+\sum_j p_{ij}b_{ij}^\prime, \text{ for some } i.$$
This definition does not depend on $i$. Clearly $b\equiv a$ (mod
$\mathfrak{m}^n$). We need to show that $\mbf{f}(b)=0$. This is
true, since $\mbf{f}(b)\equiv \mbf{f}(a_i)\equiv 0$ (mod $P_i$) for
all $i$, and $\cap_i P_i=0$.
\end{proof}


\section{Explicit Calculations Of Artin Functions}


\subsection{Monomial ideals}


We shall show that over a regular local ring the Artin functions
associated to systems of monomial ideals are actually linear.

Let $(R,\mathfrak{m})$ be a local domain and
$\mbf{X}=X_1,\ldots,X_N$ be indeterminates over $R$. For each
$\alpha=(i_1,i_2,\ldots,i_N)\in\mathbb{Z}^N_{\geq 0}$ define
$|\alpha|=i_1+\cdots+i_N$, $\supp (\alpha)=\{j: i_j>0\}$ and
$\mbf{X}^\alpha=X_1^{i_1}\cdots X_r^{i_N}$.


\begin{lemma}\label{max}
Let $(R,\mathfrak{m})$ be a local ring and $\alpha_1,\ldots,\alpha_k$
vectors in $\mathbb{Z}^N_{\geq 0}$. With
the notation as above we have
$$
\beta_n(\mbf{X}^{\alpha_1},\ldots,\mbf{X}^{\alpha_k})\leq\max_{1\leq i\leq k}\beta_n(\mbf{X}^{\alpha_i}),
$$
for all $n\in\mathbb{N}$.
\end{lemma}


\begin{proof}
Let $\mbf{a}=(a_1,\ldots,a_N)\in R^N $ be such that
$$
(\mbf{a}^{\alpha_1},\ldots,\mbf{a}^{\alpha_k})\in\mathfrak{m}^{\max\beta_n(\mbf{X}^{\alpha_i})}.
$$
From
$\mbf{a}^{\alpha_1}\in\mathfrak{m}^{\max\beta_n(\mbf{X}^{\alpha_i})}\subseteq\mathfrak{m}^{\beta_n(\mbf{X}^{\alpha_1})}$,
there is $j_1\in$ supp$(\alpha_1)$ such that
$a_{j_1}\in\mathfrak{m}^n$. Consider the set of all $\alpha_i$ whose
supports do not contain $j_1$ and without loss of generality we may
assume that $\alpha_2$ is one of them. Again there is $j_2\in$
supp$(\alpha_2)$ such that $a_{j_2}\in\mathfrak{m}^n$. Consider the
set of all $\alpha_i$ whose supports do not contain $j_1$ or $j_2$, and so on.
Iterating the process we can find a collection of indices
$\{j_1,j_2,\ldots,j_s\}$ for which every $\alpha_i$ has its support
containing one of these indices. Now we define
$$
b_j=
\begin{cases}
0 & \text{if $j\in\{j_1,\ldots,j_s\}$}\\
a_j& \text{otherwise}.\\
\end{cases}
$$
Clearly $\mbf{b}=(b_1,\ldots,b_N)\in R^N$ is a solution of the
system $\mbf{X}^{\alpha_1},\ldots,\mbf{X}^{\alpha_k}$ and
$\mbf{a}\equiv\mbf{b}$ (mod $\mathfrak{m}^n$).
\end{proof}


\begin{lemma}\label{1mono}

(i) Let $(R,\mathfrak{m})$ be an analytically irreducible local
domain. Then for each $\alpha\in\mathbb{Z}_{\geq 0}^N$ there is
$c\in \mathbb{N}$ not depending on $n$ such that
$\beta_n(\mbf{X}^\alpha)\leq cn$ for all $n$.

(ii) If $(R,\mathfrak{m})$ is a regular local ring then
$\beta_n(\mbf{X}^\alpha)=|\alpha|n-|\alpha|+1$ for all
$n\in\mathbb{N}$.
\end{lemma}


\begin{proof}
 (i) Since $(R,\mathfrak{m})$ is an analytically
irreducible domain, by \cite[Theorem 4.3]{swanson1} there is an
integer $c^\prime>0$ such that for every $n$ if
$uv\in\mathfrak{m}^{c^\prime n}$ then either $u\text{ or } v$
belongs to $\mathfrak{m}^n$. We claim that
$c=(c^\prime)^{N+|\alpha|}$ will satisfy the conclusion of (i).
Indeed, let $\mbf{a}=(a_1,\ldots,a_N)\in R^N$ be such that
$\mbf{a}^{\alpha}\in\mathfrak{m}^{c n}$. Then by the property of
$c^\prime$, there is $j\in\supp (\alpha)$ with
$a_j\in\mathfrak{m}^n$. Define $\mbf{b}=(a_1,\ldots,0,\ldots,a_N)$,
where $0$ is at the $j^{th}$-place,  which satisfies
$\mbf{b}^\alpha=0$ and $\mbf{a}\equiv\mbf{b}$ (mod
$\mathfrak{m}^n$).

(ii) Let $\mbf{a}\in R^N$ be such that
$\mbf{a}^\alpha\in\mathfrak{m}^{|\alpha|n-|\alpha|+1}$. Since $R$ is
regular, there is $i\in$ supp($\alpha$) such that
$a_i\in\mathfrak{m}^n$. Then define
$\mbf{b}=(a_1,\ldots,0,\ldots,a_N)$ where $0$ is at the
$i^{\text{th}}$-place. Clearly $\mbf{b}$ satisfies
$\mbf{b}^\alpha=0$ and $\mbf{a}\equiv\mbf{b}$ (mod
$\mathfrak{m}^n$). Thus we have the inequality
$\beta_n(\mbf{X}^\alpha)\leq |\alpha|n-|\alpha|+1$.

Conversely, consider a vector $\mbf{a}\in R^N$ with
$a_i\in\mathfrak{m}^{n-1}-\mathfrak{m}^n$ for all $i$. We have
$\mbf{a}^\alpha\in\mathfrak{m}^{|\alpha|n-|\alpha|}$ but there is no
solution $\mbf{b}$ such that $\mbf{a}\equiv \mbf{b}$ (mod
$\mathfrak{m}^n$). Hence
$\beta_n(\mbf{X}^\alpha)>|\alpha|n-|\alpha|$ and therefore we have
the above equality.
\end{proof}


\begin{theorem} Let $(R,\mf{m})$ be a local domain. Let
$I=(\mbf{X}^{\alpha_1},\ldots,\mbf{X}^{\alpha_k})\subset R[X_1,\ldots, X_N]$
be a monomial ideal with $|\alpha_1|\geq |\alpha_2|\geq\ldots\geq |\alpha_k|$.

(i) If $(R,\mathfrak{m})$ is not analytically irreducible then the
Artin function of the polynomial $X_1X_2$ does not exist.

(ii) If $(R,\mathfrak{m})$ is an analytically irreducible domain
then there is $c$ such that $\beta_n(I)\leq cn$ for all $n\in N$.

(iii) If $(R,\mathfrak{m})$ is a regular local ring then
$\beta_n(I)=|\alpha_s|n-|\alpha_s|+1$ for every $n\in\mathbb{N}$,
where $s=\min\{i|\supp(\alpha_i) \not\supseteq \supp (\alpha_j)\text{ for every
}j>i\}$.
\end{theorem}


\begin{proof}

(i) (cf \cite[Proposition 3.4]{swanson2}) Since $\widehat{R}$ is not a
domain, there are $\hat{u},\hat{v}\in\widehat{R}$ nonzero and
$\hat{u}\hat{v}=0$. By Krull's intersection theorem, there is $n$
such that $\hat{u},\hat{v}\notin\hat{\mathfrak{m}}^n$. Let
$\{u_k\},\{v_k\}$ be two sequences in $R$ that converge to
$\hat{u},\hat{v}$. Then there is $l$ such that
$u_k,v_k\notin\mathfrak{m}^n$ for all $k>l$. But $u_kv_k$ converges
to $uv=0$, so for each $m$ there is $k_m>l$ such that
$u_{k_m},v_{k_m}\in\mathfrak{m}^m$ but none of $u_{k_m},v_{k_m}$
belongs to $\mathfrak{m}^n$. Therefore the Artin function of
$X_1X_2$ does not exist.

(ii) Using Lemma~\ref{max} and Lemma~\ref{1mono} we have $\beta_n(I)\leq c^{N+\max_{1\leq i\leq k}|\alpha_i|}n$
for all $n$.

(iii) We will first show that
$\beta_n(\mbf{X}^{\alpha_1},\ldots,\mbf{X}^{\alpha_k})\leq
|\alpha_s|n-|\alpha_s|+1$. To prove this let $\mbf{a}\in R^N$ be
such that
$$(\mbf{a}^{\alpha_1},\ldots,\mbf{a}^{\alpha_k})\subseteq \mathfrak{m}^{|\alpha_s|n-|\alpha_s|+1}.$$
Since $|\alpha_s|n-|\alpha_s|+1\geq \beta_n({\mbf{X}^{\alpha_i}})$
for all $i\geq s$, we can find $\mbf{b}\in R^N$ such that
$\mbf{b}^{\alpha_i}=0$ for all $i\geq s$ and $\mbf{a}\equiv\mbf{b}$
mod $\mathfrak{m}^n$. For $i<s$ we have $\supp(\alpha_i)\supseteq
\supp(\alpha_s)$. Thus $\mbf{b}$ is also a solution for the whole
system $\{\mbf{X}^{\alpha_i}, i=1\ldots, k\}$.

 It remains to show that $\beta_n(\mbf{X}^{\alpha_1},\ldots,\mbf{X}^{\alpha_k})> |\alpha_s|n-|\alpha_s|$. For each $j\in\supp (\alpha_s)$
 choose $a_j\in\mathfrak{m}^{n-1}-\mathfrak{m}^n$. Since $\supp (\alpha_t)\not\subset\supp (\alpha_s)$ for every $t>s$,
 there is $j_t\in\supp (\alpha_t)-\supp (\alpha_s)$. Choose $ a_{j_{s+1}}=\ldots=a_{j_k}\in\mathfrak{m}^{|\alpha_s|(n-1)}$ arbitrarily.
 For $i\notin\{\supp (\alpha_s)\}\cup  \{j_{s+1},\ldots,j_k\}$ let $a_i$ be arbitrary. Then it is
clearly that
$(\mbf{a}^{\alpha_s},\ldots,\mbf{a}^{\alpha_k})\subseteq\mathfrak{m}^{|\alpha_s|(n-1)}$.
By the definition of $s$, for each $i<s$ there is $t\geq s$ such
that $\supp(\alpha_i)\supseteq\supp(\alpha_t)$. If
$\supp(\alpha_i)\supseteq\supp(\alpha_s)$ then
$\mbf{a}^{\alpha_i}\in\mathfrak{m}^{|\alpha_s|(n-1)}$ because
$|\alpha_i|\geq|\alpha_s|$. If
$\supp(\alpha_i)\supseteq\supp(\alpha_t)$ with $t\geq s+1$ then we
still have  $\mbf{a}^{\alpha_i}\in\mathfrak{m}^{|\alpha_s|(n-1)}$
since $j_t\in\supp(\alpha_t)$ and
$a_{j_t}\in\mathfrak{m}^{|\alpha_s|(n-1)}$. This says that
$(\mbf{a}^{\alpha_1},\ldots,\mbf{a}^{\alpha_k})\subseteq\mathfrak{m}^{|\alpha_s|(n-1)}$.
Then in this case for every true solution $\mbf{b}$ of the system
$\mbf{X}^{\alpha_1},\ldots,\mbf{X}^{\alpha_k}$ we have
$\mbf{a}\not\equiv\mbf{b}$ (mod $\mathfrak{m}^n$). Therefore we have
proved the converse inequality and then
$\beta_n(\mbf{X}^{\alpha_1},\ldots,\mbf{X}^{\alpha_k})=\beta_n(\mbf{X}^{\alpha_2},\ldots,\mbf{X}^{\alpha_k})=|\alpha_s|n-|\alpha_s|+1$.
\end{proof}


\subsection{Determinantal ideals}


We shall compute the Artin functions associated to any determinantal
ideal over a DVR and show that it is actually a linear function.


\begin{theorem}
Let $(R,t)$ be a discrete valuation ring. Let
$\mbf{X}=\{X_{ij}| i=1\ldots k,j=1\ldots l\}$ be a set of variables
over $R$ with $k\leq l$. Given $1\leq r\leq k$, denote by $I_r$ the ideal of $R[\mbf{X}]$
generated by all $r$-minors of the matrix $(X_{ij})_{i,j}$. Then for every $n\in\mathbb{N}$ we have
$$
\beta_n(I_r)=rn-r+1.
$$
\end{theorem}


\begin{proof}
We prove this theorem by induction on $r$. The case $r=1$ is
trivially true. Suppose we have proved for ideals generated by
$(r-1)$-minors of all matrices $(X_{ij})$ of any size. We shall
first prove that $\beta_n(I_r)\leq rn-r+1$. Let $\mbf{a}=(a_{ij})\in
R^{kl}$ be an approximate solution of order $rn-r+1$ of the defining
equations of $I_r$. That means
$$
I_r
\begin{pmatrix}
a_{11}&a_{12} & \ldots &a_{1l}\\
a_{21}& a_{22} & \ldots & a_{2l}\\
\vdots & \vdots & \vdots&\vdots\\
a_{k1} & a_{k2} & \ldots & a_{kl}\\
\end{pmatrix} \subset (t)^{rn-r+1}.
$$
Let $v$ be the (discrete) valuation associated to $(R,t)$ and
suppose $v(a_{11})=\min_{i,j}\{v(a_{ij}), a_{ij}\neq 0\}$.  We wish
to find a solution $\mbf{b}=(b_{ij})\in R^{kl}$ such that
$\mbf{a}\equiv\mbf{b}$ (mod $(t)^nR^{kl}$). If $v(a_{11})\geq n$
then clearly we can take $\mbf{b}=0$. Suppose $v(a_{11})\leq n-1$.
Note that we have $a_{ij}a_{11}^{-1}\in R$ for all $i,j$. Then we
have
\begin{eqnarray*}I_r
&\begin{pmatrix}
1& a_{12}a_{11}^{-1}& \ldots & a_{1l}a_{11}^{-1}\\
a_{21}& a_{22} & \ldots & a_{2l}\\
\vdots & \vdots & \vdots&\vdots\\
a_{k1} & a_{k2} & \ldots & a_{kl}\\
\end{pmatrix}\subset (t)^{rn-r+1-v(a_{11})}\subseteq (t)^{(r-1)n-(r-1)+1}\\
\Longrightarrow & I_r
\begin{pmatrix}
1& a_{12}a_{11}^{-1}& \ldots & a_{1l}a_{11}^{-1}\\
0& a_{22}-a_{12}a_{21}a_{11}^{-1} & \ldots & a_{2l}-a_{1l}a_{21}a_{11}^{-1}\\
\vdots & \vdots & \vdots&\vdots\\
0 & a_{k2}-a_{12}a_{k1}a_{11}^{-1}& \ldots & a_{kl}-a_{1l}a_{k1}a_{11}^{-1}\\
\end{pmatrix}\subset (t)^{(r-1)n-(r-1)+1}\\
\Longrightarrow &I_{r-1}
\begin{pmatrix}
a_{22}-a_{12}a_{21}a_{11}^{-1} & \ldots & a_{2l}-a_{1l}a_{21}a_{11}^{-1}\\
a_{32}-a_{12}a_{31}a_{11}^{-1} & \ldots & a_{3l}-a_{1l}a_{31}a_{11}^{-1}\\
\vdots & \vdots&\vdots\\
a_{k2}-a_{12}a_{k1}a_{11}^{-1}& \ldots & a_{kl}-a_{1l}a_{k1}a_{11}^{-1}\\
\end{pmatrix}\subset (t)^{(r-1)n-(r-1)+1}.
\end{eqnarray*}
By induction there is $\mbf{b}^\prime=(b^\prime_{ij})_{i,j\geq 2}\in
R^{(k-1)(l-1)}$ such that
$$
I_{r-1}
\begin{pmatrix}
b^\prime_{22} & b^\prime_{23} & \ldots & b^\prime_{2l}\\
b^\prime_{32}& b^\prime_{33} & \ldots & b^\prime_{3l}\\
\vdots & \vdots & \vdots&\vdots\\
b^\prime_{k2} & b^\prime_{k3} & \ldots & b^\prime_{kl}\\
\end{pmatrix} =0,
$$
and $a_{ij}-a_{1j}a_{i1}a_{11}^{-1}\equiv b^\prime_{ij}$ (mod
$(t)^n$) for all $i,j\geq 2$. Define
$$
b_{ij}=
\begin{cases}
a_{ij} \text{ if } i=1 \text{ or } j=1\\
b^\prime_{ij} + a_{1j}a_{i1}a_{11}^{-1} \text{ if } i,j > 1.
\end{cases}
$$
It is clearly that $a_{ij}\equiv b_{ij}$ (mod $(t)^n$) for all
$i,j$. We claim that $\mbf{b}=(b_{ij})_{i,j}$ is a solution of
$I_r$. Indeed,
$$
\mbf{b}=
\begin{pmatrix}
a_{11}&a_{12} & \ldots & a_{1l}\\
a_{21}& b^\prime_{22}+a_{12}a_{21}a_{11}^{-1} & \ldots & b^\prime_{2l}+a_{1l}a_{21}a_{11}^{-1}\\
a_{31}& b^\prime_{32}+a_{12}a_{31}a_{11}^{-1} & \ldots & b^\prime_{3l}+a_{1l}a_{31}a_{11}^{-1}\\
\vdots&\vdots & \vdots&\vdots\\
a_{k1}& b^\prime_{k2}+a_{12}a_{k1}a_{11}^{-1}& \ldots & b^\prime_{kl}+a_{1l}a_{k1}a_{11}^{-1}\\
\end{pmatrix}
$$
is the sum of two matrices
$$
\begin{pmatrix}
0 & 0 & 0& \ldots & 0\\
0 & b^\prime_{22} & b^\prime_{23} & \ldots & b^\prime_{2l}\\
0 & b^\prime_{32} & b^\prime_{33} & \ldots & b^\prime_{3l}\\
\vdots & \vdots & \vdots&\vdots\\
0 & b^\prime_{k2} & b^\prime_{k3} & \ldots & b^\prime_{kl}\\
\end{pmatrix}
\text{ and }
\begin{pmatrix}
a_{11} & a_{12} & \ldots & a_{1l}\\
a_{21} & a_{12}a_{21}a_{11}^{-1} & \ldots & a_{1l}a_{21}a_{11}^{-1}\\
a_{31} & a_{12}a_{31}a_{11}^{-1} & \ldots & a_{1l}a_{31}a_{11}^{-1}\\
\vdots & \vdots & \vdots&\vdots\\
a_{k1} & a_{12}a_{k1}a_{11}^{-1}& \ldots & a_{1l}a_{k1}a_{11}^{-1}\\
\end{pmatrix},
$$
which are of rank at most $r-2$ and $1$ (resp). Thus rank($\mbf{b}$)
is at most $r-1$, which means all of its $r$-minors are 0, as
desired.

Now we shall show that $\beta_n(I_r)> rn-r$. Consider the following
matrix
$$
\mbf{a}=(a_{ij})_{i,j}=
\begin{pmatrix}
t^{n-1} & 0 & \ldots & 0 & \ldots & 0\\
0 & t^{n-1} & \ldots & 0 &\ldots & 0 \\
\vdots & \vdots & \ldots & \vdots & \ldots & \vdots\\
0 & 0 & \ldots & t^{n-1} & \ldots & 0\\
\vdots & \vdots & \ldots & \vdots & \ldots & \vdots\\
0 & 0 & \ldots & . & \ldots & 0
\end{pmatrix}\in R^{kl},
$$
where $a_{ii}=t^{n-1}$ for $i=1,\ldots,r$ and $a_{ij}=0$ otherwise. It
is an approximate solution of order $(t)^{rn-r}$. Let
$\mbf{b}=(b_{ij})_{i,j}$ be any solution of $I_r$. If
$\mbf{a}\equiv\mbf{b}$ (mod $(t)^n$) then $v(b_{ij})\geq n$ for all
$i\neq j$, and $v(b_{ii})= n-1$. Consider the submatrix
$(b_{ij})_{1\leq i,j\leq r}$. It has the determinant
$$
\det (b_{ij})_{1\leq i,j\leq r}=b_{11}\ldots b_{rr}+\sum_{\delta\in S_r,\delta\neq id}(-1)^{sgn(\delta)}b_{1\delta(1)}\ldots b_{r\delta(r)},
$$
which is not equal to 0 since $v(b_{11}\ldots b_{rr})=n(r-1)$ and
$v(b_{1\delta(1)}\ldots b_{r\delta(r)})> n(r-1)$ for all
$\delta\neq$ id, a contradiction. Thus $\beta_n(I_r)> rn-r$. Hence
the equality $\beta_n(I_r)=rn-n+1$ follows.
\end{proof}


\begin{remark}\label{second_example} There is no hope to have the linearity of the Artin
functions of determinantal ideals over rings of dimension at least 3, due to a counterexample
by Rond in ~\cite{rond2}. However the question on the linearity of the Artin function of the
polynomial $XY-ZW$ over the power series ring in 2 variables $k[T_1,T_2]$ is still unknown.
\end{remark}


\noindent{\textbf{Acknowledgment.}} This project was started and
most of it was carried out when I was working under the
direction of Irena Swanson at New Mexico State University. I would
like to thank her for introducing me to the theory of Artin
approximation and for many helpful conversations and for her
ceaseless encouragement. I would also like to thank Paul C. Roberts and the referee for many
suggestions to improve the presentation of this paper.



\vspace{.3in}

\noindent \textsc{\small Department of Mathematics, University of Utah,
Salt Lake City, UT
84112}

\noindent {\small Email address: \texttt{dinh@math.utah.edu}.}


\end{document}